\documentclass[10pt]{article}
\usepackage{latexsym}
\usepackage{amsmath, amsthm, amsfonts}
\usepackage{graphics}
\usepackage{graphicx}
\usepackage{color}
\usepackage[table,xcdraw]{xcolor}
\usepackage{float}
\usepackage{array}
\usepackage{hyperref}

\newcommand{\abs}[1]{\left\lvert#1\right\rvert}
\newtheorem{theorem}{Theorem}[section]
\newtheorem{lemma}[theorem]{Lemma}

\newtheorem{conjecture}{Conjecture}[section]

\newtheorem{remark}{Remark}[section]
\def\smallskip{\addvspace{\smallskipamount}}
\def\medskip{\addvspace{\medskipamount}}
\def\bigskip{\addvspace{\bigskipamount}}
\def\makefootline{\baselineskip=24pt \line{\the\footline}}
\def\pagecontents{\ifvoid\topins\else\unvbox\topins\fi
   \dimen@=\dp255 \unvbox255
   \ifvoid\footins\else
      \vskip\skip\footins \footnoterule \unvbox\footins\fi
     \ifr@ggedbottom \kern-\dimen@ \vfil \fi}
\def\footnoterule{\kern-3pt\hrule width 2truein \kern 2.6pt}

\begin{document}

\title{Complex Dynamics of a Second Order Rational Difference Equation}

\author{Sk. Sarif Hassan and Anupam Bhandari\\
  \small {Department of Mathematics, College of Engineering Studies,}\\
  \small {University of Petroleum and Energy Studies,}\\
  \small {Bidholi, Dehradun, India.}\\
  \small Emails: {\texttt{\textcolor[rgb]{0.00,0.07,1.00}{s.hassan@ddn.upes.ac.in,} and \texttt{\textcolor[rgb]{0.00,0.07,1.00}{a.bhandari@ddn.upes.ac.in}}}}\\
}

\maketitle
\begin{abstract}
\noindent The dynamics of the second order rational difference equation $\displaystyle{z_{n+1}=\frac{\alpha + z_{n-1}}{\beta z_n + z_{n-1}}}$ with the real parameter $\alpha$, $\beta$ and arbitrary non-negative real initial conditions is investigated a decade ago. In the present manuscript, the same has been revisited considering the parameters $\alpha$ and $\beta$ as complex numbers and the initial values as arbitrary complex numbers. It is found that some of the results which are valid in real line but does not valid in complex plane. The chaotic solutions of the difference equation with complex parameters are achieved, however there does not exists such solutions in the case of real parameters. \\
\end{abstract}

\begin{flushleft}\footnotesize
{Keywords: Rational difference equation, Local asymptotic stability, Chaotic trajectory and Periodicity. \\

}.
\end{flushleft}

\section{Introduction and Background}

Consider the second order rational difference equation

\begin{equation}
\displaystyle{z_{n+1}=\frac{\alpha + z_{n-1}}{\beta z_n + z_{n-1}}},  n=0,1,2,\ldots
\label{equation:total-equationA}
\end{equation}%
where the parameters $\alpha$, $\beta$ and the initial conditions $z_{-1}$  and $z_{0}$ are arbitrary complex numbers.

\addvspace{\bigskipamount}
\noindent
This rational difference equation Eq.(\ref{equation:total-equationA}) is studied considering the parameters $\alpha$ and $\beta$ as real numbers and the initial conditions as non-negative real numbers in \cite{C-L}, \cite{C-V}, \cite{P-Y} \& \cite{Ca-Ch-L-Q-1}. In this manuscript, it is an attempt to understand the dynamics in the complex plane. Alike work is done for other second order rational difference equations in \cite{S-S}, \cite{S-H} \& \cite{S-E}. Applications of such kind of dynamical systems have been studied in \cite{1-1} \& \cite{2-2}.\\

\noindent
Here, a very brief review of the difference equation Eq.(\ref{equation:total-equationA}) in real line is adumbrated \cite{C-L}. The results are as follows: \\

\begin{itemize}
  \item The equilibrium of the equation Eq.(\ref{equation:total-equationA}) is locally asymptotically stable when $\beta < 1+ 4\alpha$ and unstable and, more precisely, a saddle point equilibrium when $\beta > 1+ 4\alpha$.
  \item Equation Eq.(\ref{equation:total-equationA}) has a prime period-two solution ($\phi$ and $\psi$) if and only if $\beta > 1+ 4\alpha$. Furthermore, $\beta > 1+ 4\alpha$ holds, the period-two solution is ‘‘unique’’ and the values of $\phi$ and $\psi$ are the positive roots of the quadratic equation $t^2-t+\frac{\alpha}{\beta-1}=0$.
  \item Assume $\beta \leq 1+ 4\alpha$, then the equilibrium of Eq.(\ref{equation:total-equationA}) is a global attractor.
\end{itemize}

\noindent
What is still an open problem in the real set up is as follows: \\
\noindent
Assume that $\beta > 1+ 4\alpha$ holds. Investigate the basin of attraction of the prime period two cycle \cite{K-L} \& \cite{Ko-L}.\\

\noindent
Here our main purpose is to study the dynamics of Eq. (\ref{equation:total-equationA}) under the condition that the parameters and the initial conditions are arbitrary complex numbers.

\section{Local Asymptotic Stability of the Equilibriums}

The equilibrium points of Eq.(\ref{equation:total-equationA}) are the solutions of the quadratic equation
\[
\bar{z}=\frac{\alpha+\bar{z}}{\beta \bar{z}+\bar{z}}
\]
The Eq.(\ref{equation:total-equationA}) has the two equilibria points
$ \bar{z}_{1,2} =\frac{1-\sqrt{1+4 \alpha +4 \alpha  \beta }}{2 (1+\beta )}$ and $\frac{1+\sqrt{1+4 \alpha +4 \alpha  \beta }}{2 (1+\beta )}$ respectively.
\noindent
The linearized equation of the rational difference equation Eq.(\ref{equation:total-equationA}) with respect to the equilibrium point $ \bar{z}_{1} = \frac{1-\sqrt{1+4 \alpha +4 \alpha  \beta }}{2 (1+\beta )}$ is

\begin{equation}
\label{equation:linearized-equation}
\displaystyle{
z_{n+1} + \frac{\beta}{1+\beta } z_{n} + \frac{1+2 \alpha +\sqrt{1+4 \alpha +4 \alpha  \beta }}{2 \alpha +2 \alpha  \beta } z_{n-1}=0,  n=0,1,\ldots
}
\end{equation}

\noindent
with associated characteristic equation
\begin{equation}
\lambda^{2} + \frac{\beta}{1+\beta } \lambda + \frac{1+2 \alpha +\sqrt{1+4 \alpha +4 \alpha  \beta }}{2 \alpha +2 \alpha  \beta } = 0.
\end{equation}

\addvspace{\bigskipamount}

\noindent

\begin{lemma}
The zeros of a quadratic polynomial $\lambda^2-r\lambda-s=0$ lie inside unit disk in $\mathbb{C}$ if $\abs{r} < \abs{1-s} <2$.
\end{lemma}

\noindent
The following result gives the local asymptotic stability of the equilibrium $\bar{z}_{1}$ of the Eq.(\ref{equation:total-equationA}).

\begin{theorem}
The equilibriums $\bar{z}_{1}=\frac{1-\sqrt{1+4 \alpha +4 \alpha  \beta }}{2 (1+\beta )}$ of the Eq.(\ref{equation:total-equationA}) is locally asymptotically stable if $$\abs{1+\frac{1+2 \alpha +\sqrt{1+4 \alpha +4 \alpha  \beta }}{2 \alpha +2 \alpha  \beta }}<2$$
\end{theorem}

\addvspace{\bigskipamount}

\begin{proof}
The equilibriums $\bar{z}_{1}=\frac{1-\sqrt{1+4 \alpha +4 \alpha  \beta }}{2 (1+\beta )}$ of the Eq.(\ref{equation:total-equationA}) is \emph{locally asymptotically stable} if the modulus of the both zeros of the characteristic equation (3) are less than $1$. By the Lemma $2.1$, the condition for making the zeros lying inside the unit disk is $\abs{\frac{\beta}{1+\beta }} < \abs{1+\frac{1+2 \alpha +\sqrt{1+4 \alpha +4 \alpha  \beta }}{2 \alpha +2 \alpha  \beta }} < 2$ since $\abs{\frac{\beta}{1+\beta }}<1$ is obvious for any complex number $\beta$. Now if $\abs{1+\frac{1+2 \alpha +\sqrt{1+4 \alpha +4 \alpha  \beta }}{2 \alpha +2 \alpha  \beta }} < 2$ then it is obvious that $\abs{\frac{\beta}{1+\beta }} < \abs{1+\frac{1+2 \alpha +\sqrt{1+4 \alpha +4 \alpha  \beta }}{2 \alpha +2 \alpha  \beta }}$. Therefore the condition boils down to $$\abs{1+\frac{1+2 \alpha +\sqrt{1+4 \alpha +4 \alpha  \beta }}{2 \alpha +2 \alpha  \beta }}<2$$

\end{proof}

\noindent
It is observed that the minimum value of the $\abs{1+\frac{1+2 \alpha +\sqrt{1+4 \alpha +4 \alpha  \beta }}{2 \alpha +2 \alpha  \beta }}$ is $1.66614$ which is less than $2$ when $\alpha=-0.82781+0.224354i$ $(\abs{\alpha}<1)$ and $\beta=0.492467-0.333602i$ $(\abs{\beta}<1)$. \\
\noindent
This numerical observation makes a guarantee that there are parameters $\alpha$ and $\beta$ in the unit disk such that the conditional inequality does hold good. Therefore existence of parameters is ensured for local asymptotic stability of the equilibrium of the difference equation Eq.(\ref{equation:total-equationA}).\\\\

\noindent
The linearized equation of the rational difference equation Eq.(\ref{equation:total-equationA}) with respect to the equilibrium point $ \bar{z}_{2} = \frac{1+\sqrt{1+4 \alpha +4 \alpha  \beta }}{2 (1+\beta )}$ is

\begin{equation}
\label{equation:linearized-equation}
\displaystyle{
z_{n+1} + \frac{\beta}{1+\beta} z_{n} + \frac{1+2 \alpha -\sqrt{1+4 \alpha +4 \alpha  \beta }}{2 \alpha +2 \alpha  \beta } z_{n-1}=0,  n=0,1,\ldots
}
\end{equation}

\noindent
with associated characteristic equation
\begin{equation}
\lambda^{2} + \frac{\beta}{1+\beta} \lambda +\frac{1+2 \alpha -\sqrt{1+4 \alpha +4 \alpha  \beta }}{2 \alpha +2 \alpha  \beta } = 0.
\end{equation}

\addvspace{\bigskipamount}

\begin{theorem}
The equilibriums $\bar{z}_{2}=\frac{1+\sqrt{1+4 \alpha +4 \alpha  \beta }}{2 (1+\beta )}$ of the Eq.(\ref{equation:total-equationA}) is locally asymptotically stable if $$\abs{1+\frac{1+2 \alpha -\sqrt{1+4 \alpha +4 \alpha  \beta }}{2 \alpha +2 \alpha  \beta }}<2$$
\end{theorem}

\begin{proof}
Proof is similar to the proof of the Theorem $2.2$.\\
\end{proof}

\noindent
It is found that the minimum value of the $\abs{1+\frac{1+2 \alpha -\sqrt{1+4 \alpha +4 \alpha  \beta }}{2 \alpha +2 \alpha  \beta }}$ is $0.834925$ which is less than $2$ when $\alpha=0.04008-0.237697i$ $(\abs{\alpha}<1)$ and $\beta=0.598157+0.0345986i$ $(\abs{\beta}<1)$. \\
\noindent
This observation ensures the existence of the parameters $\alpha$ and $\beta$ in the unit disk such that the conditional inequality does hold good. Consequently, the existence of the local asymptotic stability of the equilibrium $\bar{z}_2$ of the difference equation Eq.(\ref{equation:total-equationA}) is ensured.

\begin{lemma}
A necessary and sufficient condition for one root of $\lambda^2-r\lambda-s=0$ to have modulus less than one and the other root to have modulus greater than one is $\abs{r} > \abs{1-s}$.\\\\
In this case, $\lambda$ is called a saddle-point equilibrium.
\end{lemma}

\begin{theorem}
The equilibrium $\bar{z}_{\pm}$ is unstable, more precisely a saddle point equilibrium if $\abs{\frac{\beta}{1+\beta}} > \abs{1+\frac{1+2 \alpha \pm \sqrt{1+4 \alpha +4 \alpha  \beta }}{2 \alpha +2 \alpha  \beta }}$.\\
\end{theorem}

\begin{proof}
The proof follows from the Lemma $2.4$.\\
\end{proof}

\noindent
It is observed that the maximum value of the $\abs{\frac{\beta}{1+\beta}} - \abs{1+\frac{1+2 \alpha \pm \sqrt{1+4 \alpha +4 \alpha  \beta }}{2 \alpha +2 \alpha  \beta }}$ is $0.959948$ when $\alpha=0.00794746+0.0120667i$ and $\beta=1.94598+7.32387i$, that is $\abs{\frac{\beta}{1+\beta}} > \abs{1+\frac{1+2 \alpha \pm \sqrt{1+4 \alpha +4 \alpha  \beta }}{2 \alpha +2 \alpha  \beta }}$ is holding well for $\alpha=0.00794746+0.0120667i$ and $\beta=1.94598+7.32387i$. This observation suggests that there are $\alpha$ and $\beta$ such that the solution is unstable about the equilibriums.\\

\noindent
Consider the parameters of the difference equation Eq.(\ref{equation:total-equationA}) $\alpha=25+22i$, $\beta=67+85i$. Here it is noted that $\abs{\beta}<\abs{1+4\alpha}$. The one of the equilibriums is $0.553877 - 0.051776i$. The linearized equation about the equilibrium $0.553877 - 0.051776i$ is $\lambda^2-r\lambda-s=0$ where $r=3.05424 - 0.661201i$ and $s=-0.0386782 + 0.015489i$. Here it is $\abs{r}>\abs{1-s}$. Therefore the equilibrium $0.553877 - 0.051776i$ is unstable (saddle point). The trajectories are given in the following Fig.$1$ which clearly depict the instability nature of the equilibrium $0.553877 - 0.051776i$.

\begin{figure}[H]
      \centering

      \resizebox{14cm}{!}
      {
      \begin{tabular}{c c}
      \includegraphics [scale=6]{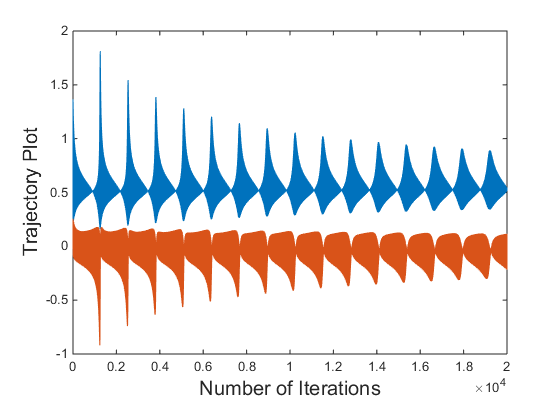}
      \includegraphics [scale=6]{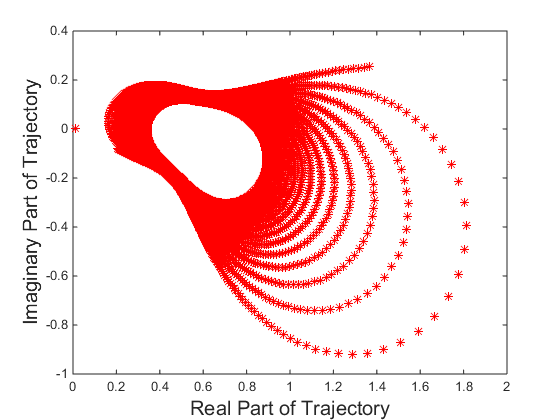}\\
      \includegraphics [scale=6]{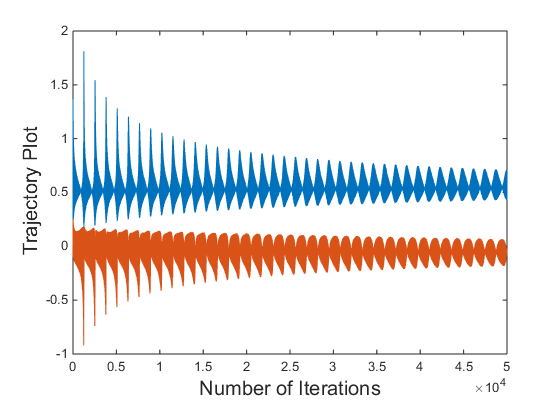}
      \includegraphics [scale=6]{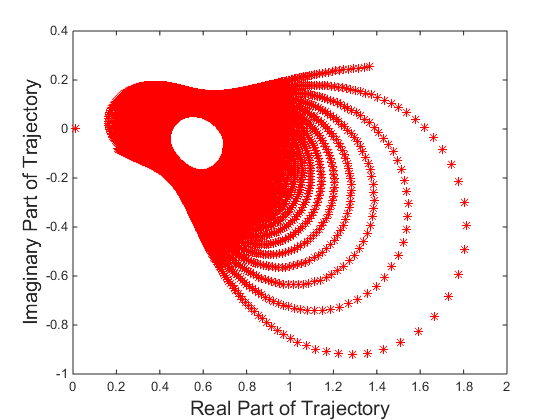}\\
      \includegraphics [scale=6]{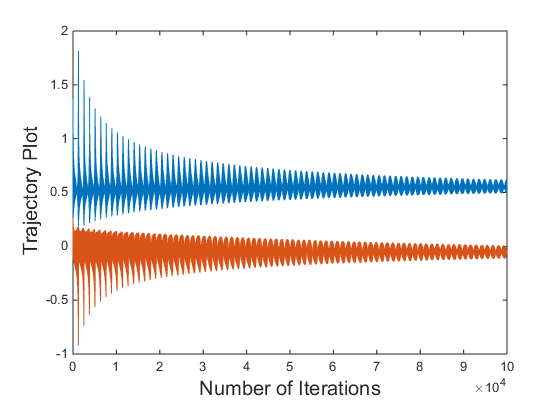}
      \includegraphics [scale=6]{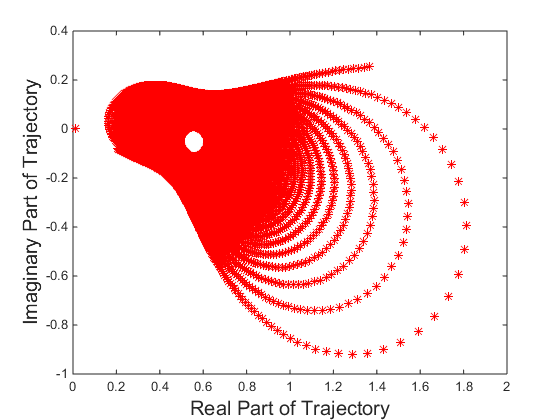}\\

            \end{tabular}
      }
\caption{Trajectory Plots.}
      \begin{center}

      \end{center}
      \end{figure}

\noindent
In Fig.1, the trajectory plots for $20000$, $50000$ and  $100000$ iterations respectively are given. it is seen in the Fig. 1 that the trajectory plots are very unstable in nature in converging to the equilibrium $0.553877 - 0.051776i$.

%
%

\subsection{A Special case $\alpha=\beta$}

When the parameters $\alpha$ and $\beta$ are equal, we shall see the local stability of the equilibriums. \\

\noindent
The equilibriums of the Eq.(\ref{equation:total-equationA}) when $\alpha=\beta$ are $1$ and $\frac{\alpha }{1+\alpha }$. It can be easily seen that the equilibrium $1$ is \emph{locally asymptotically stable} for all $\alpha$. \\\\
\noindent
The linearized equation of the rational difference equation Eq.(\ref{equation:total-equationA}) with respect to the equilibrium point $\frac{\alpha }{1+\alpha }$ is

\begin{equation}
\label{equation:linearized-equation}
\displaystyle{
z_{n+1} + \frac{2+\alpha }{1+\alpha } z_{n} + \frac{1}{\alpha +\alpha ^2} z_{n-1}=0,  n=0,1,\ldots
}
\end{equation}

\noindent
with associated characteristic equation
\begin{equation}
\lambda^{2} + \frac{2+\alpha }{1+\alpha } \lambda + \frac{1}{\alpha +\alpha ^2} = 0.
\end{equation}

\addvspace{\bigskipamount}

\begin{theorem}
The equilibriums $\frac{\alpha }{1+\alpha }$ of the Eq.(\ref{equation:total-equationA}) where $\alpha=\beta$ is locally asymptotically stable if $$\abs{1+\frac{1}{1+\alpha }} < \abs{1+\frac{1}{\alpha +\alpha ^2}}<2$$
\end{theorem}

\addvspace{\bigskipamount}

\begin{proof}
The proof is very similar to the proof of the theorem Theorem $2.2$.\\
\end{proof}

\bigskip

\begin{remark}
Consider $\alpha=0.530797553008973 + 0.779167230102011i$ and $\beta=4.670053421145915 + 1.299062084737301i$ where $\abs{1+4\alpha}=4.412249115813187$ and $\abs{\beta}=4.847366424808288$, i.e. $\abs{\beta}>\abs{1+4\alpha}$. For any initial values, the trajectory is convergent and converges to $0.464833509611819+0.121170658272098i$. The corresponding trajectory plot is given in Fig. $2$.
\end{remark}

\begin{figure}[H]
      \centering

      \resizebox{16cm}{!}
      {
      \begin{tabular}{c}
      \includegraphics [scale=5]{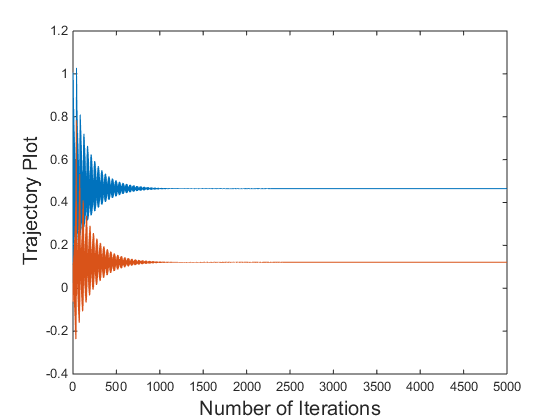}
      \includegraphics [scale=5]{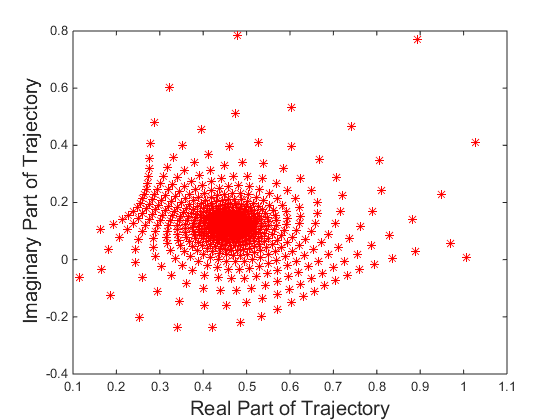}\\
      \end{tabular}
      }
\caption{Trajectory Plots.}
      \begin{center}

      \end{center}
      \end{figure}

\noindent
In the real parameters $\alpha$ and $\beta$, it is found that the equilibrium is unstable (saddle point) if $\abs{\beta}>\abs{1+4\alpha}$. and the trajectory would be periodic of prime period 2 as stated in section $1$. \\ But in the complex case, there exist complex parameters $\alpha$ and $\beta$ with $\abs{\beta}>\abs{1+4\alpha}$ where the trajectory is convergent and converges to one of the equilibriums which is neither non-trivial prime period 2 solution nor unstable. This ensures that the result obtained in the real parameters is no more valid in the complex set up.

\section{Unbounded Solutions}
Here we shall investigate the unboundedness of the solution of the difference equation Eq.(\ref{equation:total-equationA}). To proceed, we would try start asking the following question.\\

\noindent
Let $\epsilon >0$, For what values of $\alpha$, $\beta$ $\in \mathbb{C}$ does the following hold:

  $$\abs{\frac{\alpha+z_{n-1}}{\beta z_n+z_{n-1}}}<\epsilon$$ for all $z_n$ and $z_{n-1}$ $\in \mathbb{C}$ with $\abs{z_n}<\epsilon$ and $\abs{z_{n-1}}<\epsilon$.\\

\noindent
The answer to this question would confirm the parameters $\alpha$ and $\beta$ such that the trajectory would be bounded.
\noindent
First let us assume that $\beta \neq 0$. Then we can choose $\abs{z_{n-1}}=\frac{\beta}{m}$ and $z_n=-\frac{1}{m}$ where $m$ is so large that $\abs{z_n}<\epsilon$ and $\abs{z_{n-1}}<\epsilon$, then $\beta z_n+z_{n-1}=0$ so that the fraction becomes infinite. So $\beta$ must be zero and the fraction reduced to $\frac{\alpha +z_{n-1}}{z_{n-1}}$ which is infinite if $z_{n-1}=0$ and $\alpha \neq 0$. \\

\noindent
Therefore $\alpha=\beta=0$ which is a trivial solution only if $\epsilon>1$. For $\epsilon\leq1$, there is no solution at all. \\

\noindent
This observation ensures that there exist infinitely many unbounded solutions of difference equation Eq.(\ref{equation:total-equationA}).

\section{Periodic of Solutions}

\label{section:periodicity}

A solution $\{z_n\}_n$ of a rational difference equation is said to be \emph{globally periodic} of period $t$ if $z_{n+t}=z_n$ for any given initial conditions. A solution $\{z_n\}_n$ is said to be \emph{periodic with prime period} $p$ if p is the smallest positive integer having this property.\\

\noindent
We shall look for the prime period two solutions of the difference equation Eq.(\ref{equation:total-equationA}) and its corresponding local stability analysis.

\subsection{Prime Period Two Solutions}

Let $\ldots, \phi ,~\psi , ~\phi , ~\psi ,\ldots$, $\phi  \neq \psi $ be a prime period two solution of the difference equation $z_{n+1}=\frac{\alpha + z_{n-1}}{\beta z_n + z_{n-1}}$. Then $\phi=\frac{\alpha+\phi}{\beta \psi+\phi}$ and $\psi=\frac{\alpha+\psi}{\beta \phi+\psi}$. This two equations lead to the set of solutions (prime period two) except the equilibriums as $\left\{\phi \to 0.5 -\frac{0.5 \sqrt{\alpha  (4 -4 \beta )+(1 -\beta )^2}}{-1+\beta },\psi \to 0.5 +\frac{0.5 \sqrt{\alpha  (4 -4 \beta )+(1 -\beta )^2}}{-1+\beta }\right\}$.\\

\noindent
Let $\ldots, \phi ,~\psi , ~\phi , ~\psi ,\ldots$, $\phi  \neq \psi $ be a prime period two solution of the equation Eq.(\ref{equation:total-equationA}). We set $$u_n=z_{n-1}$$ $$v_n=z_{n}$$
\noindent
Then the equivalent form of the difference equation Eq.(\ref{equation:total-equationA}) is
$$u_{n+1}=z_n$$ $$v_{n+1}=\frac{\alpha+z_{n-1}}{\beta z_n+z_{n-1}}$$
\noindent
Let T be the map on $\mathbb{C}\times\mathbb{C}$ to itself defined by $$T\left(
                                                                           \begin{array}{c}
                                                                             u \\
                                                                             v \\
                                                                           \end{array}
                                                                         \right)=\left(
                                                                                    \begin{array}{c}
                                                                                      v \\
                                                                                      \frac{\alpha+u}{\beta v+u} \\
                                                                                    \end{array}
                                                                                  \right)
                                                                         $$
\noindent
Then $\left(
                                                                           \begin{array}{c}
                                                                             \phi \\
                                                                             \psi \\
                                                                           \end{array}
                                                                         \right)$ is a fixed point of $T^2$, the second iterate of $T$. \\

$$T^2\left(
                                                                           \begin{array}{c}
                                                                             u \\
                                                                             v \\
                                                                           \end{array}
                                                                         \right)=\left(
                                                                                    \begin{array}{c}
                                                                                      \frac{\alpha+u}{\beta v+u} \\\\
                                                                                      \frac{\alpha+v}{\beta \frac{\alpha+u}{\beta v+u} +v } \\
                                                                                    \end{array}
                                                                                  \right)
                                                                         $$

$$T^2\left(
                                                                           \begin{array}{c}
                                                                             u \\
                                                                             v \\
                                                                           \end{array}
                                                                         \right)==\left(
                                                                                    \begin{array}{c}
                                                                                      g(u,v) \\
                                                                                      h(u,v) \\
                                                                                    \end{array}
                                                                                  \right)
                                                                         $$

\noindent
 where $g(u,v)=\frac{\alpha+u}{\beta v+u}$ and $h(u,v)=\frac{\alpha+v}{\beta \frac{\alpha+u}{\beta v+u} +v }$. Clearly the two cycle is locally asymptotically stable when the eigenvalues of the
Jacobian matrix $J_{T^2}$, evaluated at $\left(
                                                                           \begin{array}{c}
                                                                             \phi \\
                                                                             \psi \\
                                                                           \end{array}
                                                                         \right)$ lie inside the unit disk.\\

\noindent
We have, $$J_{T^2}\left(
                                                                           \begin{array}{c}
                                                                             \phi \\
                                                                             \psi \\
                                                                           \end{array}
                                                                         \right)=\left(
                                                                                   \begin{array}{cc}
                                                                                     \frac{\delta g}{\delta u }(\phi, \psi) & \frac{\delta g}{\delta v}(\phi, \psi) \\\\
                                                                                     \frac{\delta h}{\delta u }(\phi, \psi) & \frac{\delta h}{\delta v }(\phi, \psi) \\
                                                                                   \end{array}
                                                                                 \right)
                                                                         $$

\noindent
\\
\\
where $\frac{\delta g}{\delta u }(\phi, \psi)=\frac{-\alpha +\beta  \psi }{(\phi +\beta  \psi )^2}$ and $\frac{\delta g}{\delta v }(\phi, \psi)=\frac{\beta  (\alpha +\psi ) (\alpha -\beta  \psi )}{\left(\alpha  \beta +\phi  \psi +\beta  \left(\phi +\psi ^2\right)\right)^2}$\\\\

$\frac{\delta h}{\delta u }(\phi, \psi)=-\frac{\beta  (\alpha +\phi )}{(\phi +\beta  \psi )^2}$
and $\frac{\delta h}{\delta v }(\phi, \psi)=\frac{\psi +\frac{\beta  (\alpha +\phi )}{\phi +\beta  \psi }-(\alpha +\psi ) \left(1-\frac{\beta ^2 (\alpha +\phi )}{(\phi +\beta  \psi )^2}\right)}{\left(\psi +\frac{\beta  (\alpha +\phi )}{\phi +\beta  \psi }\right)^2}$
\\ \\
\noindent
Now, set $$\chi=\frac{\delta g}{\delta u }(\phi, \psi)+\frac{\delta h}{\delta v }(\phi, \psi)=\frac{-\alpha +\beta  \psi }{(\phi +\beta  \psi )^2}+\frac{\psi +\frac{\beta  (\alpha +\phi )}{\phi +\beta  \psi }-(\alpha +\psi ) \left(1-\frac{\beta ^2 (\alpha +\phi )}{(\phi +\beta  \psi )^2}\right)}{\left(\psi +\frac{\beta  (\alpha +\phi )}{\phi +\beta  \psi }\right)^2}$$ $$\lambda=\frac{\delta g}{\delta u }(\phi, \psi) \frac{\delta h}{\delta v }(\phi, \psi)-\frac{\delta g}{\delta v }(\phi, \psi) \frac{\delta h}{\delta u }(\phi, \psi)=\frac{(-\beta  \phi +\alpha  (\phi +\beta  (-1+\psi ))) (\alpha -\beta  \psi )}{(\phi +\beta  \psi ) \left(\alpha  \beta +\phi  \psi +\beta  \left(\phi +\psi ^2\right)\right)^2}$$

\bigskip

\bigskip

\noindent
By the Linear Stability theorem, the prime period two solutions $\phi$ and $\psi$ would be locally asymptotically stable if the condition ($|\chi| < 1+ |\lambda| < 2$) holds well.\\
In particular, for the prime period $2$ solution, $\left\{\phi \to 0.5 -\frac{0.5 \sqrt{\alpha  (4 -4 \beta )+(1 -\beta )^2}}{-1+\beta },\psi \to 0.5 +\frac{0.5 \sqrt{\alpha  (4 -4 \beta )+(1 -\beta )^2}}{-1+\beta }\right\}$, we shall see the local asymptotic stability for some example cases of parameters $\alpha$ and $\beta$. The general form of $\chi$ and $\lambda$ would be very complected.\\
\noindent
Consider the prime period two solution of the difference equation Eq.(\ref{equation:total-equationA}), $\phi \to 2+3i,\psi \to 1+2i$ corresponding two the parameters $\alpha \to 1$ and $\beta \to 1+i$. \\
\noindent
In this case, $\abs{\chi}=0.00186987$ and $\abs{\lambda}=0.174209$. Therefore the condition ($0.00186987 < 1.174209 < 2$) the prime period $2$ solution $\phi \to 1.30024 + 0.624811i,\psi \to -0.300243 - 0.624811i$ is \emph{locally asymptotically stable}.\\

\noindent
Consider another example case where $\alpha=51+8i$ and $\beta=26+80i$. The prime period $2$ solution of the Eq.(\ref{equation:total-equationA}) are $\phi=1.01487 + 0.536364i$ and $\psi=-0.0148677 - 0.536364i$. Here $\abs{\chi}=0.00013967$ and $\abs{\lambda}=0.996087$.  Therefore the condition ($0.00013967 < 1.996087 < 2$) the prime period $2$ solution $\phi=1.01487 + 0.536364i$ and $\psi=-0.0148677 - 0.536364i$ is \emph{locally asymptotically stable}. The corresponding trajectory plot is given in Fig. 3.

\begin{figure}[H]
      \centering

      \resizebox{16cm}{!}
      {
      \begin{tabular}{c c}
      \includegraphics [scale=4]{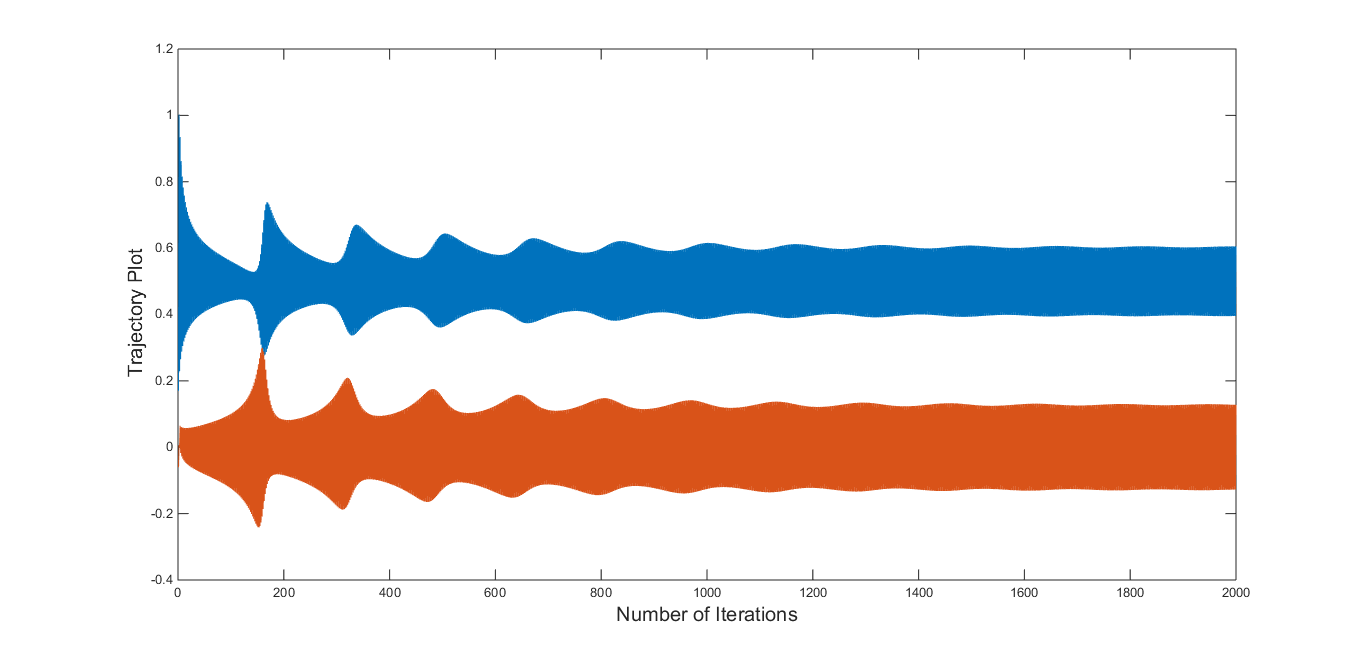}
      \includegraphics [scale=6.2]{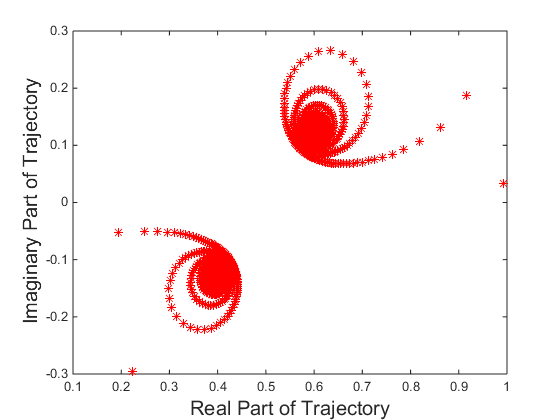}\\
      \end{tabular}
      }
\caption{Prime Period Two Trajectory of the equation (1).}
      \begin{center}

      \end{center}
      \end{figure}

\noindent
Computationally, it is seen that there does not exist any periodic solution of the difference equation Eq.(\ref{equation:total-equationA}) of period greater than 3. Hence the following conjecture has been made.\\

\begin{conjecture}
There does not exist any periodic solution of the difference equation Eq.(\ref{equation:total-equationA}) $p \geq 3$.
\end{conjecture}

\noindent
\\
In the case of real parameters $\alpha$ and $\beta$, it is an open problem to determine the basin of attraction of the prime period two cycle when $\beta > 1+ 4\alpha$ holds. In the complex parameters a computational study has been made in gathering the prime period two solutions which is given in the following Table 1. Also the prime period two solutions of $52$ different cases are plotted in the Fig. $4$.

\begin{figure}[H]
      \centering

      \resizebox{9cm}{!}
      {
      \begin{tabular}{c}
      \includegraphics [scale=4]{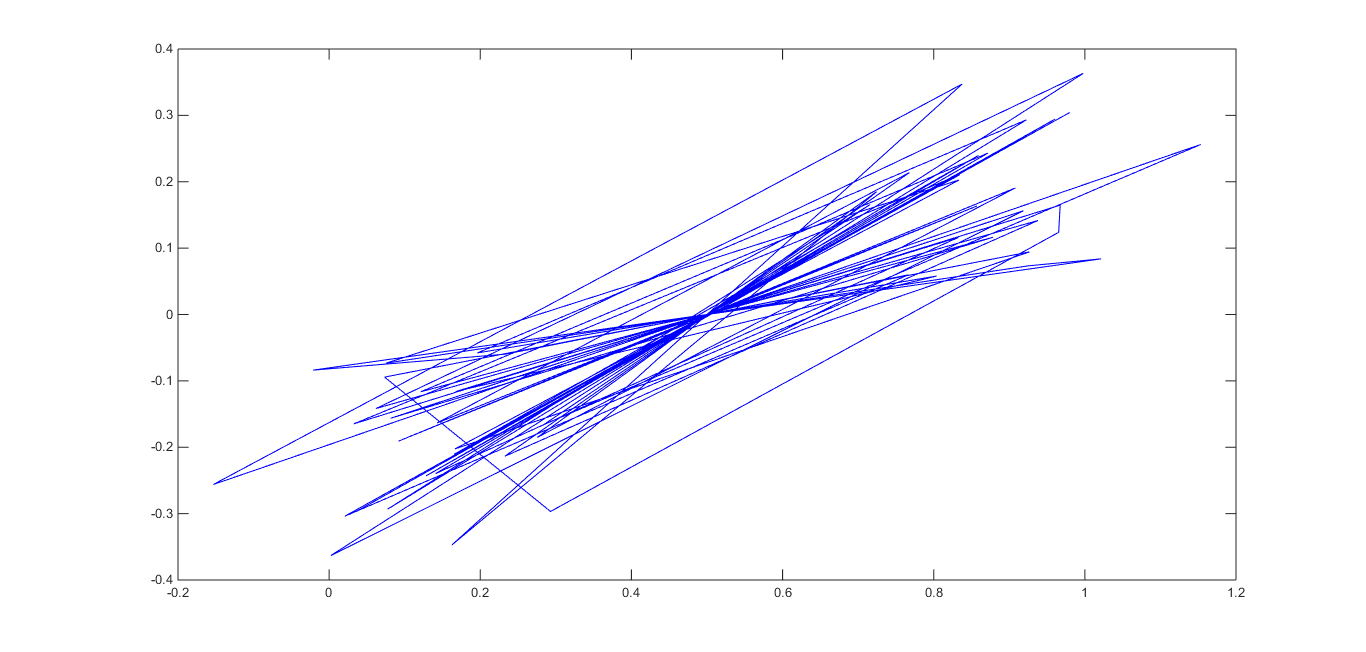}
      \end{tabular}
      }
\caption{Chaotic Trajectories of the equation Eq.(\ref{equation:total-equationA}) of four different cases as stated in Table 1.}
      \begin{center}

      \end{center}
      \end{figure}

\noindent
When the parameters $\alpha$ and $\beta$ such that $\abs{\beta}>\abs{1+4\alpha}$ holds, then the prime period two solutions are the zeros of the quadratic polynomial $t^2-t+\frac{\alpha}{\beta-1}$ which has been seen computationally.

\begin{table}[H]

\begin{tabular}{| m{5cm}||m{5cm}| |m{5cm}|}
\hline
\centering   \textbf{Parameters} $\alpha$, $\beta$ &
\begin{center}
\textbf{$\abs{\beta}>\abs{1+4\alpha}$}
\end{center}
 &
\begin{center}
\textbf{$\phi$ and $\psi$}
\end{center} \\
\hline
\centering $\alpha=0.6855 + 0.2941i$, $\beta=1.06125 + 2.49727i$ & \begin{center}
$\checkmark$
\end{center}
&
\begin{center}
$\phi=0.03921 - 0.29456i$, $\psi=0.96078 + 0.29456i$
\end{center}\\
\hline
\centering $\alpha=0.42264 + 0.35960i$, $\beta=1.116 + 2.2276i $ &
\begin{center}
$\checkmark$
\end{center}
 &
\begin{center}
$\phi=
0.1282 - 0.24313i$, $\psi=0.871775 + 0.24313i$
\end{center}\\
\hline
\centering $\alpha=0.290185 + 0.31752i$, $\beta=1.30738 + 2.8708i$ & \begin{center}
$\checkmark$
\end{center}
&
\begin{center}
$\phi=0.12117 - 0.116448i$, $\psi=0.878827 + 0.116448i
$
\end{center} \\
\hline
\centering $\alpha=0.1806 + 0.04505i$, $\beta=1.446346 + 1.04231i$ & \begin{center}
$\checkmark$
\end{center} &
\begin{center}
$\phi=0.08145 - 0.15624i$, $\psi=0.918547 + 0.15624i$
\end{center}\\
\hline
\centering $\alpha=0.78287 + 0.69378i$, $\beta=0.019604 + 2.5296i$ & \begin{center}
$\checkmark$
\end{center} &
\begin{center}
$\phi=0.00229 - 0.36314i$, $\psi=0.9977 + 0.36314i$
\end{center}\\
\hline
\centering $\alpha=0.3389 + 0.2101i$, $\beta=1.020305 + 2.71909i$ & \begin{center}
$\checkmark$
\end{center} &
\begin{center}
$\phi=0.06199 - 0.141634i$, $\psi=0.938 + 0.14163i$
\end{center}\\
\hline
\centering $\alpha=0.50128 + 0.43172i$, $\beta=1.99512 + 2.4348i$ & \begin{center}
$\checkmark$
\end{center} &
\begin{center}
$\phi=0.2324 - 0.2136i$, $\psi=0.7675 + 0.2136i$
\end{center}\\
\hline
\centering $\alpha=0.2815 + 0.23038i$, $\beta=1.4222 + 1.87371i$ & \begin{center}
$\checkmark$
\end{center} &
\begin{center}
$\phi=0.14301 - 0.16331i$, $\psi=0.85698 + 0.16331i$
\end{center}\\
\hline
\centering $\alpha=0.45134 + 0.2409i$, $\beta=1.43009 + 2.5685i$ & \begin{center}
$\checkmark$
\end{center} &
\begin{center}
$\phi=0.09193 - 0.190715i$, $\psi=0.90806 + 0.19071i$
\end{center}\\
\hline
\centering $\alpha=0.1386 + 0.5882i$, $\beta=0.732313 + 2.42027i$ & \begin{center}
$\checkmark$
\end{center} &
\begin{center}
$\phi=0.2754 - 0.1851i$, $\psi=0.72455 + 0.1851i$
\end{center}\\
\hline
\hline
\centering $\alpha=0.44944 + 0.96353i$, $\beta=0.084595 + 2.91887i$ & \begin{center}
$\checkmark$
\end{center} &
\begin{center}
$\phi=0.16239 - 0.3472i$, $\psi=0.8376 + 0.3472i$
\end{center}\\
\hline
\end{tabular}
\caption{Prime Period Two Solutions of the equation Eq.(\ref{equation:total-equationA}) for different choice of parameters.}
\label{Table:}
\end{table}

\section{Chaotic Solutions}
This is something which is absolutely new feature of the dynamics of the difference equation Eq.(\ref{equation:total-equationA}) which did not arise in the real set up of the same difference equation. Computationally we have encountered some chaotic solutions of the difference equation Eq.(\ref{equation:total-equationA}) for some parameters $\alpha$ and $\beta$ which are given in the following Table $2$. \\\\
\noindent
In this present study, the largest Lyapunov exponent is calculated for a given solution of finite length numerically \cite{Wolf} to show the trajectories are chaotic.\\\\
\noindent
From computational evidence, it is arguable that for complex parameters $\alpha$ and $\beta$ which are stated in the following Table $2$, the solutions are chaotic for every initial values.\\\\

\begin{table}[H]

\begin{tabular}{| m{5cm}||m{4cm}| |m{5cm}|}
\hline
\centering   \textbf{Parameters} $\alpha$, $\beta$ &
\begin{center}
\textbf{$\abs{\beta}<\abs{1+4\alpha}$}
\end{center}
 &
\begin{center}
\textbf{Lyapunav exponent}
\end{center} \\
\hline
\centering $\alpha=0.096455 + 0.13197i$, $\beta=0.94205 + 0.95613i$ & \begin{center}
$\checkmark$
\end{center}
&
\begin{center}
$1.3215$
\end{center}\\
\hline
\centering $\alpha=0.8235 + 0.175i$, $\beta=0.32713 + 1.9979i $ &
\begin{center}
$\checkmark$
\end{center}
 &
\begin{center}
$1.021$
\end{center}\\
\hline
\centering $\alpha=0.7195 + 0.9961i$, $\beta=0.70906 + 2.9137i$ & \begin{center}
$\checkmark$
\end{center}
&
\begin{center}
$0.5564$
\end{center} \\
\hline
\centering $\alpha=0.5747 + 0.3260i$, $\beta=0.9128 + 2.1413i$ & \begin{center}
$\checkmark$
\end{center} &
\begin{center}
$1.373$
\end{center}\\
\hline
\centering $\alpha=0.8322 + 0.61739i$, $\beta=1.04025 + 2.5916i$ & \begin{center}
$\checkmark$
\end{center} &
\begin{center}
$1.0655$
\end{center}\\
\hline
\centering $\alpha=0.140255 + 0.2601i$, $\beta=0.1736 + 1.288i$ & \begin{center}
$\checkmark$
\end{center} &
\begin{center}
$1.0985$
\end{center}\\
\hline
\centering $\alpha=0.5114 + 0.0606i$, $\beta=1.45137 + 1.6696i$ & \begin{center}
$\checkmark$
\end{center} &
\begin{center}
$1.8997$
\end{center}\\
\hline
\centering $\alpha=0.8954 + 0.5825i$, $\beta=1.165 + 2.564i$ & \begin{center}
$\checkmark$
\end{center} &
\begin{center}
$0.7655$
\end{center}\\
\hline
\centering $\alpha=0.9720 + 0.0314i$, $\beta=1.6708 + 2.5071i$ & \begin{center}
$\checkmark$
\end{center} &
\begin{center}
$0.7658$
\end{center}\\
\hline
\centering $\alpha=0.8989 + 0.3536i$, $\beta=0.2403 + 1.7073i$ & \begin{center}
$\checkmark$
\end{center} &
\begin{center}
$1.7325$
\end{center}\\
\hline
\end{tabular}
\caption{Chaotic solutions of the equation Eq.(\ref{equation:total-equationA}) for different choice of parameters and initial values.}
\label{Table:}
\end{table}

\noindent
 The chaotic trajectory plots including corresponding complex plots of four examples whose parameters are given in Table 1 starting from top row are given the following Fig. $5$.

\begin{figure}[H]
      \centering

      \resizebox{16cm}{!}
      {
      \begin{tabular}{c c}
      \includegraphics [scale=4]{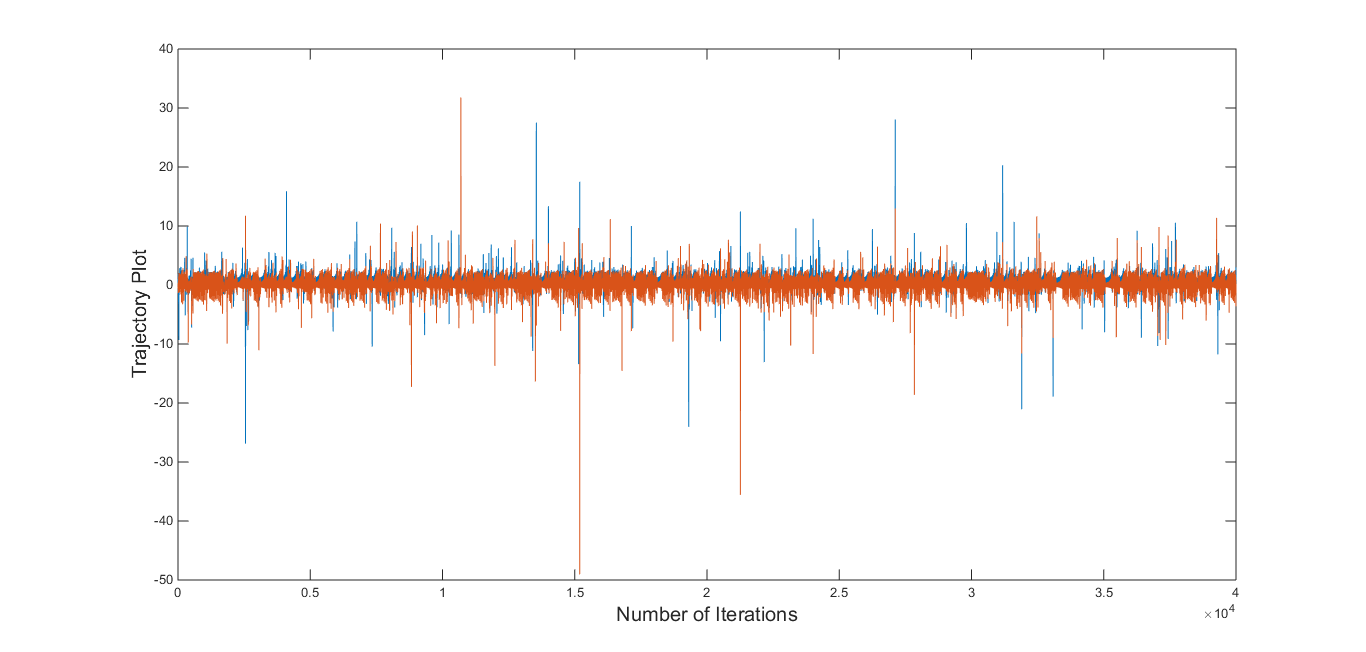}
      \includegraphics [scale=6.2]{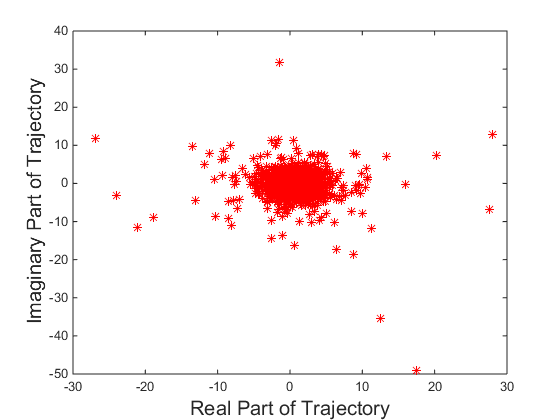}\\
      \includegraphics [scale=4]{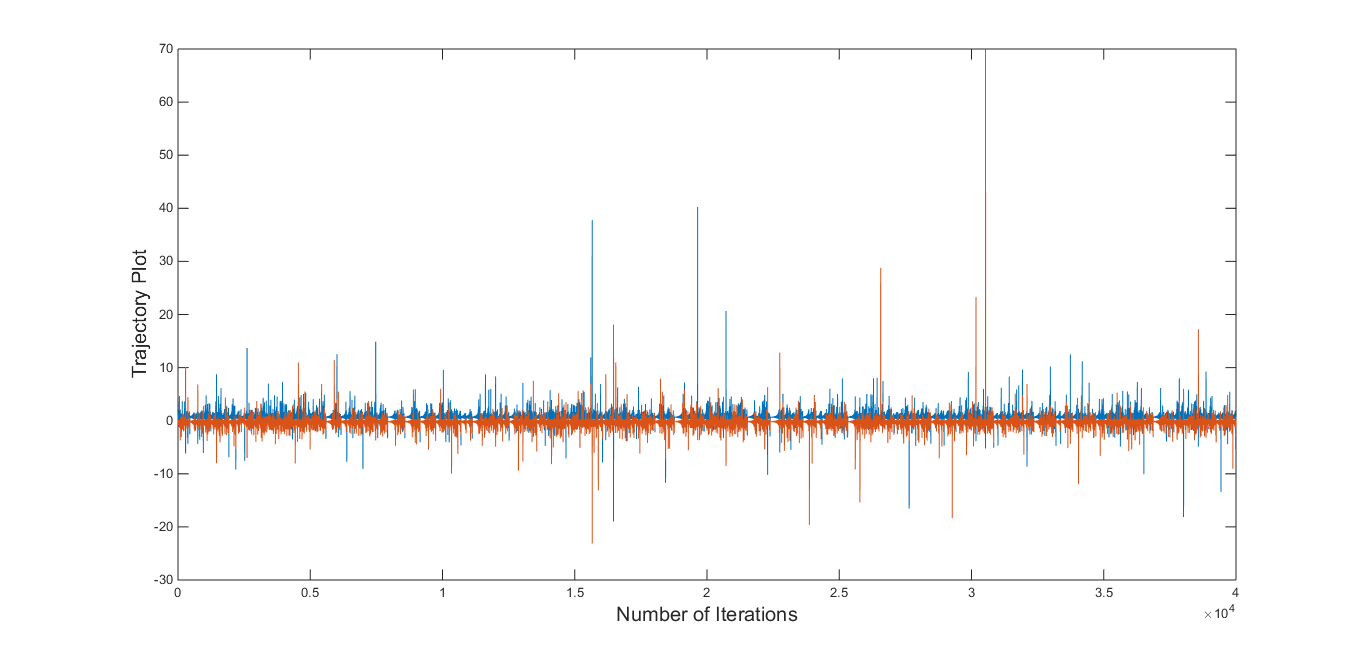}
      \includegraphics [scale=6.2]{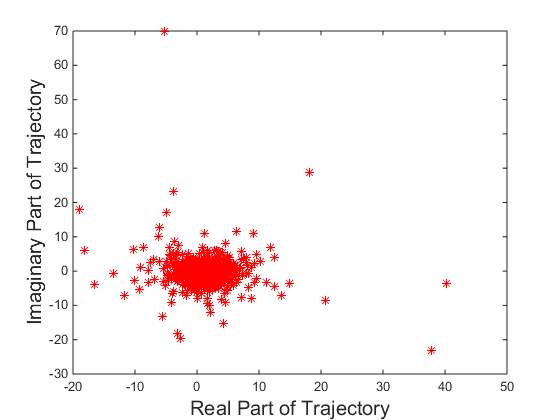}\\
      \includegraphics [scale=4]{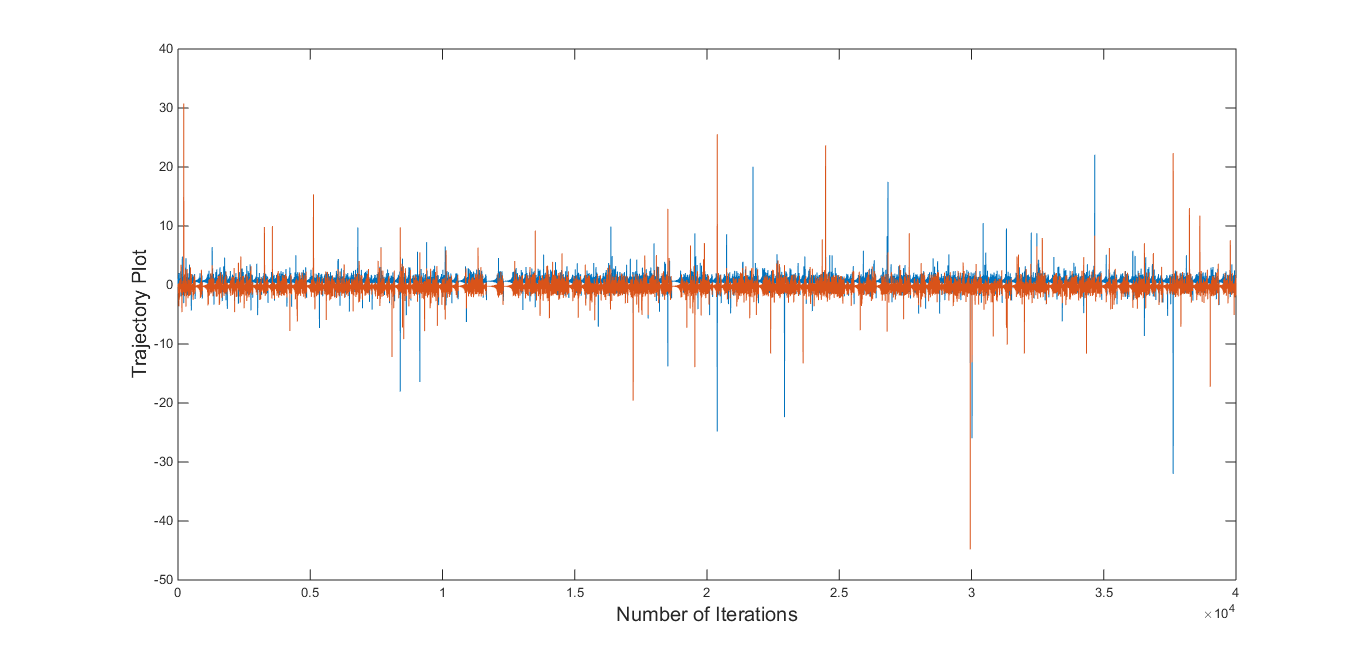}
      \includegraphics [scale=6.2]{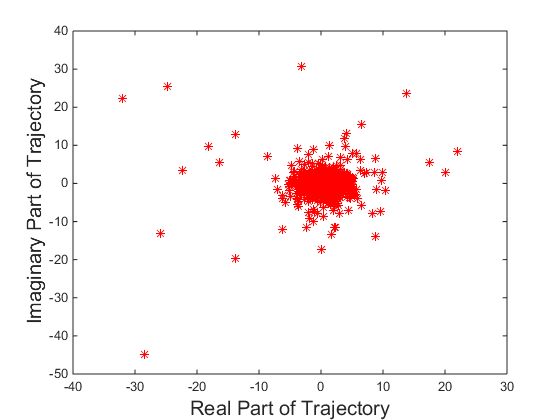}\\
      \includegraphics [scale=4]{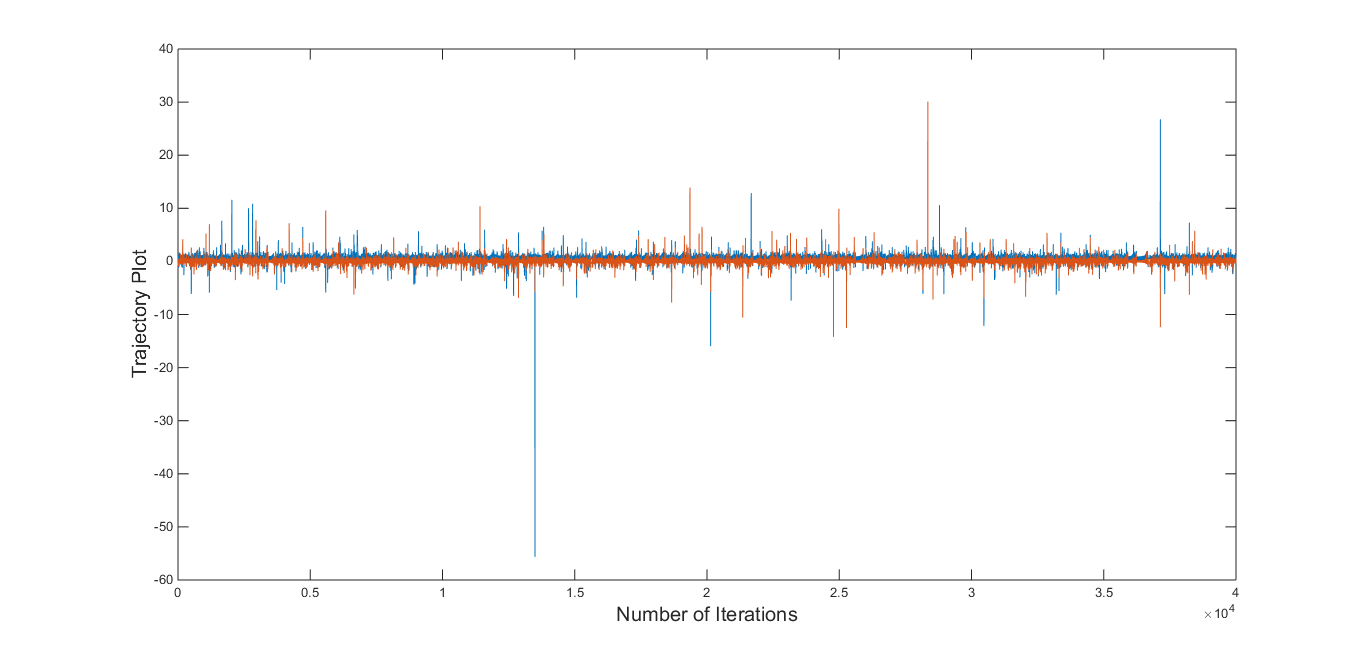}
      \includegraphics [scale=6.2]{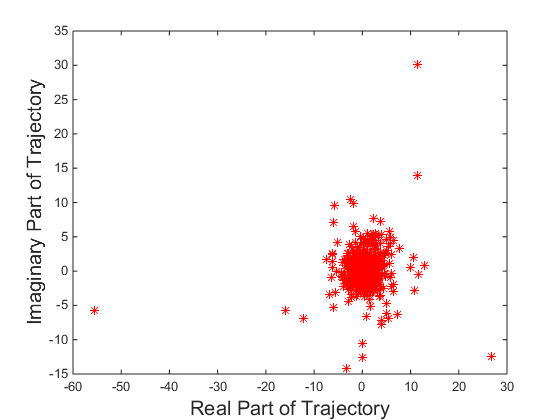}\\
      \end{tabular}
      }
\caption{Chaotic Trajectories of the equation Eq.(\ref{equation:total-equationA}) of four different cases as stated in Table 1.}
      \begin{center}

      \end{center}
      \end{figure}
\noindent
In the Fig. 5, for each of the four cases ten different initial values are taken and plotted in the left and in the right corresponding complex plots are given. From the Fig. 5, it is evident that for the four different cases the basin of the chaotic attractor is neighbourhood of the centre $(0, 0)$ of complex plane.
\noindent
In the Table $2$, a few example cases of chaotic trajectories are given where it is observed and noted that $\abs{\beta}<\abs{1+4\alpha}$ condition holds but in the case of real parameters $\alpha$ and $\beta$, $\beta<1+4\alpha$ was the condition for local asymptotic stability of equilibrium. In this regard, a conjecture has been made.

\begin{conjecture}
The chaotic solutions of the difference equation Eq.(\ref{equation:total-equationA}) exist if $\abs{\beta}<\abs{1+4\alpha}$ holds.

\end{conjecture}
\section{Future Endeavors}

In continuation of the present work the study of the difference equation ${z_{n+1}=\frac{\alpha_n+z_{n-1}}{\beta_n z_n + z_{n-1}}}$ where $\alpha_n$, $\beta_n$, are all convergent sequence of complex numbers and converges to $\alpha$, $\beta$ respectively is indeed would be very interesting and that we would like to pursue further. Also the most generalization of the present rational difference equation with delay terms is $${z_{n+1}=\frac{\alpha + z_{n-l}}{\beta z_{n-k} +z_{n-l}}}$$ where $l$ and $k$ are delay terms and it demands similar analysis which we plan to pursue in near future. The similar technique can also be used to solve nonlinear coupled differential equations which arise in various engineering problems.



\end{document}